\documentclass[]{amsart}
\usepackage{amsfonts, amsmath} 
\usepackage{amssymb}
\usepackage{amsthm}
\usepackage{mathrsfs}
\usepackage[english]{babel}
\usepackage[latin1]{inputenc}
\usepackage[T1]{fontenc} 
\usepackage{color}
\begin{document}
		
		\newtheorem{definition}{Definition}[section]
		\newtheorem{theorem}{Theorem}[section]
		\newtheorem{lemma}{Lemma}[section]
		\newtheorem*{T. lemma}{Technical lemma }
		\newtheorem{proposition}{Proposition}[section]
		\numberwithin{equation}{section}
		\newtheorem{remark}[theorem]{Remark}
		
		\title[ ]{Multiplicity of solutions for a class of elliptic problem of $p$-Laplacian type  with a $p$-Gradient term}
		
	\author{Zakariya Chaouai}
	\email{z.chaouai@gmail.com}
	\author{Soufiane Maatouk}
	\email{sf.maatouk@gmail.com}
	\address{Center of Mathematical Research and Applications of Rabat (CeReMAR), Laboratory of Mathematical Analysis and Applications (LAMA), Department of Mathematics, Faculty of Sciences, Mohammed V University, P.O. Box 1014, Rabat, Morocco.}	
\begin{abstract}
	 We consider the following problem
		$$(P) 
		\begin{cases}
		-\Delta_{p}u= c(x)|u|^{q-1}u+\mu |\nabla u|^{p}+h(x) & \ \ \mbox{ in }\Omega,\\
		u=0 & \ \ \mbox{ on }   \partial\Omega,  
		\end{cases}$$
	where $\Omega$ is a bounded set in $\mathbb{R}^{N}$ ($N\geq 3$) with a smooth boundary, $1<p<N$, $q>0$, $\mu \in \mathbb{R}^{*}$, and $c$ and $ h$ belong to $L^{k}(\Omega)$ for some $k>\frac{N}{p}$. In this paper, we assume that  $c\gneqq 0$ a.e. in $\Omega$ and $h$ without sign condition, then 
	 we prove the existence of at least two bounded solutions  under the condition that  $\|c\|_{k}$ and $\|h\|_{k}$ are suitably small. For this  purpose, we use the Mountain Pass theorem, on an equivalent problem to $(P)$  with variational structure. Here, the main difficulty  is that the nonlinearity term considered does not satisfy Ambrosetti and Rabinowitz condition. The key idea   is to replace the former condition by the   \textbf{nonquadraticity condition at infinity}.
\end{abstract}
\maketitle
\section{Introduction and main result}
	Let $\Omega$ be a bounded set in $\mathbb{R}^{N}$ ($N\geq 3$)  with a smooth boundary $\partial \Omega$.	In this paper, we are concerned with  the following  elliptic problem
	$$(P) 
	\begin{cases}
	-\Delta_{p}u= c(x)|u|^{q-1}u+\mu |\nabla u|^{p}+h(x) & \ \ \mbox{ in }\Omega,\\
	u=0 & \ \ \mbox{ on }   \partial\Omega,  
	\end{cases}$$
	where $\Delta_{p}u:=\mbox{div}(|\nabla u|^{p-2}\nabla u)$ is the $p$-Laplacian operator, $1<p<N$, $q>0$, $\mu \in \mathbb{R}^{*}$, and $c$  and $ h$ belong to $L^{k}(\Omega)$ for some $k>\frac{N}{p}$.  
	
	In the literature, there are many results concerning the existence,  the uniqueness, and the multiplicity of solutions for  models like  $(P)$ under various assumptions on $c$ and $h$. At first, it is important to  mention that the sign of $c$ plays a crucial role in the problem $(P)$ regarding uniqueness, as well as existence, of bounded solutions. In this setting, we refer to  (\cite{Sirakov2}) for more details.    In the coercive case, that is $c(x)\leq -\alpha_{0}$ a.e. in $\Omega$\; for some $\alpha_{0}>0$, Boccardo, Murat and Puel (\cite{Boccardo,Puel,Murat}), proved the existence of bounded solutions for more general divergence form problems with quadratic growth in the gradient by using the sub and supersolution method. Moreover, Barles and Murat ([6]) and Barles et at. ([5]) have treated the uniqueness question for similar problems.  Notice that, if we allow  $c(x)\leq 0$ a.e. in $\Omega$, then Ferone and Murat (\cite{Ferone-Murat1},\cite{Ferone-Murat2}) observed that finding solutions to  $(P)$ becomes rather complex without  imposing some strong regularity conditions on the data. For the particular case  $c\equiv 0$, there had been many contributions (\cite{Abdellaoui, Maderna, Poretta}). However, for $c\leq 0$ that  may vanish only on some parts of $\Omega$, the uniqueness of solutions was left open until the recent paper  authored by Arcoya et at. (\cite{Coster}). This last result was proved for $p=2,$ $ q=1$, and under the following condition
	$$\begin{cases}
		c, h \mbox{ belong to } L^{k}(\Omega) \mbox{ for some } k>\frac{N}{2}, \mu \in L^{\infty}(\Omega) \mbox{ and } meas(\Omega \backslash Supp\; c)>0,&\\
	\displaystyle {\inf_{u\in W_{c}, \|u\|_{H^{1}_{0}(\Omega)}}}\int_{\Omega}\left(|\nabla u|^{2}-\|\mu^{+}\|_{L^{\infty}(\Omega)}h^{+}(x)u^{2}\right)>0, & \\
	\displaystyle {\inf_{u\in W_{c}, \|u\|_{H^{1}_{0}(\Omega)}}}\int_{\Omega}\left(|\nabla u|^{2}-\|\mu^{-}\|_{L^{\infty}(\Omega)}h^{-}(x)u^{2}\right)>0.
	\end{cases}$$
	where $W_{c}:=\{w\in H^{1}_{0}(\Omega): \; c(x)w(x)=0, \mbox{ a.e. in }  \Omega \}$. For a related uniqueness result see also Arcoya et at. (\cite{Coster2}).  
	
	The case where  $c(x)\gneqq 0$ a.e. in $\Omega$, the question of non-uniqueness  has been being an open problem  given by Sirakov (\cite{Sirakov}) and it has received considerable attention by many authors. Moreover, it should be pointed out  that the sign of $h$ and whether $\mu$ is a function or  a constant, generate additional difficulties for solving $(P)$. In this setting,  Jeanjean and  Sirakov (\cite{Sirakov2})  showed   the existence of two bounded solutions assuming that $\mu \in \mathbb{R}^{\ast}$, $c$ and  $h$ are in  $L^{k}(\Omega)$ for some  $ k>\frac{N}{2}$  and satisfying 
	$$
\|[\mu h]^{+}\|_{L^{\frac{N}{2}}(\Omega)}<C_{N},$$
$$
\max \{\|c\|_{L^{k}(\Omega)},\|[\mu h]^{-}\|_{L^{k}(\Omega)} \}<\bar{c},
$$
where $\bar{c}>0$ depends only on $N, k, meas(\Omega), |\mu|, \|[\mu h]^{+}\|_{L^{k}(\Omega)}$, and $C_{N}$ is the optimal constant in Sobolev's inequality. Here, $h$ is allowed to change sign. Shortly after, this result was extended by Coster and Jeanjean (\cite{Coster3}) for $\mu$ is a bounded function such that $\mu(x)\geq \mu_{1}>0$ by using   the degree topological method.
	
	Finally, in  the case where $c$ is allowed to change sign and  with $c(x)\gneqq 0$ a.e. in $\Omega$, Jenajean and Quoirin (\cite[Theorem 1.1]{Quoirin})  showed the existence of two bounded positive solutions  when $h\gneqq 0$, $\mu$ is a positive constant, and   $c^{+}$ and $\mu h$ are suitably small. 
	
	We would  also like to mention that all the  above quoted multiplicity results  were restricted to the Laplacian operator with quadratic growth in the gradient, i.e. $p=2$,  and for $q=1$. Moreover, it is  interesting to mention that when $c$ is allowed to change sign the solutions  are  positive. \\
	
	  In this work, we prove the multiplicity of bounded solutions for the problem $(P)$ by assuming  the following assumption 
	  	$$(H) 
	  	\begin{cases}
	  	c, h \mbox{ belongs to } L^{k}(\Omega) \mbox{ for some } k>\frac{N}{p}, h \mbox{ is allowed to change sign}, & \\
	  	 c\gneqq 0 \mbox{ a.e. in } \Omega, q>0, \mbox{ and } \mu \in \mathbb{R}^{*}. &    
	  	\end{cases}$$

	Now, we give a brief exposition of the proof of our multiplicity result. At first, without loss of generality, we solve the problem $(P)$ by  restricting it to the  case $\mu$ is a positive constant. For $\mu$ is a negative constant, we replace $u$ by $-u$ in $(P)$, then we conclude. Next, we observe that the problems of type $(P)$  do not have a variational formulation due to the presence of the $p$-gradient term. To overcome  this difficulty, we perform the Kazdan-Kramer change of variable, that is, $v=(e^{\frac{\mu u}{p-1}}-1)/\mu$.
	  	Thus, we obtain  the following equivalent problem $(P^{\prime})$  
	  	$$(P^{\prime}) 
	  	\begin{cases}
	  	-\Delta_{p}v= c(x)g(v)+ h(x)f(v) & \ \ \mbox{ in }\Omega,\\
	  	v=0 & \ \ \mbox{ on }   \partial\Omega,  
	  	\end{cases}$$ 
	  	where 
	  	\begin{equation}\label{function g}
	  g(s)=\frac{(p-1)^{q-p+1}}{\mu^{q}}(1+\mu s)^{p-1}|\ln(1+\mu s)|^{q-1}\ln(1+\mu s),\; \mbox{ with } s>\frac{-1}{\mu},
	  	\end{equation} 
	  	and 
	  	\begin{equation}\label{f}
	  	f(s)=\frac{(1+\mu s)^{p-1}}{(p-1)^{p-1}}.
	  	\end{equation}

	We mean by bounded weak solutions of  $(P^{\prime})$, the functions $v\in W_{0}^{1,p}(\Omega) \cap L^{\infty}(\Omega) $ satisfying
	$$\int_{\Omega}|\nabla v|^{p-2}\nabla v \nabla u=  \int_{\Omega}c(x)g(v)u+ \int_{\Omega}h(x)f(v)u, $$
	 for any $ u\in W_{0}^{1,p}(\Omega) \cap L^{\infty}(\Omega).$ Obviously,  if  $v>\frac{-1}{\mu}$ is a solution of  $(P^{\prime})$, then $u=\frac{p-1}{\mu}\ln (1+\mu v)$ is a solution of $(P)$. Hence, the solutions obtained here are not necessarily positive (compare with (\cite{Quoirin})).

	  One of the most fruitful ways to deal with $(P^{\prime})$ is the variational method, which takes into account that the weak solutions of  $(P^{\prime})$ are critical points in $W_{0}^{1,p}(\Omega)$  of the $C^{1}$-functional
	 
	 \begin{equation}\label{func}
	 I(v)=\frac{1}{p}\int_{\Omega}|\nabla v|^{p}-\int_{\Omega}c(x)G(v)-\int_{\Omega}h(x)F(v),
	 \end{equation}
	with $G(s)=\int_{0}^{s}g(t)dt$ and  $F(s)=\int_{0}^{s}f(t)dt$.

    In this work, to obtain the two critical points for $I$, we use the Mountain Pass Theorem  to show one critical point and  the standard lower semicontinuity argument to show the other. For the first one, according to the famous paper  by Ambrosetti and Rabinowitz (\cite{AR}), the most important step is to show that $I$ satisfies the  Palais-Smale  condition at the level $\tilde{c}$ (see Definition \ref{ Palais-Smale condition}). The fulfillment of this condition   relies on the well-known   Ambrosetti-Rabinowitz condition (($AR_{c}$) for short), namely
	$$\mbox{ there exist }  \; \theta >p \; \mbox{ and } \; s_{0}>0 \;  \mbox{ such that }\;   0<\theta G(s)\leq sg(s),\;   \mbox{ as } \;  |s|>s_{0}.$$
	
	Unfortunately, this condition is somewhat restrictive and not being satisfied by many nonlinearities $g$.  However,  many researches have been made to drop the $(AR_{c})$. We refer, for instance,  to  \cite{Costa, Zou, Miyagaki, Marcelo, Ubilla}. Notice that, the  nonlinearity $g$ considered here  does not satisfy  ($AR_{c}$). Moreover, since we do not assume any sign condition on $h$, the fulfillment of the Palais-Smale condition turns out  more delicate (see eg. \cite{Chile, Quoirin}). To the best of our knowledge, only Jenajean and Quoirin (\cite{Quoirin}),  recently, proved the Palais-Smale condition under the assumptions  $c$ changes  sign,  $h$ is positive, and without assuming ($AR_{c}$). In their proof, for $p=2$ and $q=1$, the authors based one of the arguments on the positivity of $h$ and the explicit determination of a function $H$;
	$$H(s)=g(s)s-2G(s).$$ 
   In our situation, as $h$ is allowed to change sign and the analog of their function $H$  can not be computed explicitly, due to our general consideration of $p$ and $q$ ($1<p<N$ and $q>0$), hence, their arguments can not be adapted. 
 
 The key point to show  the Palais-Smale condition in this paper is to   prove that $g$, among other conditions, satisfies the  following (see Lemma \ref{(NQ)}), 
 $$(NQ)\ \;  \; \; \; H(s)=g(s)s-pG(s)\to +\infty, \;\mbox{ where }\; s\to +\infty. $$
 
 The condition $(NQ)$ is a variant of the  well known  \textbf{nonquadraticity condition at infinity}, which  was introduced by Costa and Malgalh\~{a}es (\cite{Costa}), and is given as follows
 $$(CM) \; \; \mbox{ there exist  } a>0 \mbox{ and } \nu \geq \nu_{0}>0 \mbox{ such that } \liminf_{|s|\to \infty} \frac{H(s)}{|s|^{\nu}}\geq a.$$
 
  Observe that, since $\nu>0$,  then $(NQ)$ is weaker than $(CM)$. Moreover, it should be noted that $(NQ)$  was considered by Furtado and Silva in their recent paper (\cite{Marcelo}). Our result follows by using similar arguments.\\

  Concerning the existence of the second critical point handled by the standard lower semicontinuity argument,  we look for a local minimum in $W_{0}^{1,p}(\Omega)$ for the functional $I$. Indeed, we observe that $I$ takes positive values in a large sphere, due to its geometrical structure (see Proposition \ref{goem structural}), and $I(0)=0$.\\

  Now we state the  main result of this paper
  \begin{theorem}\label{main result}
  Assume that $(H)$ is satisfied. If  $\|c\|_{k}$ and $\|h\|_{k}$ are suitably small, then the functional $I$ has at least two critical points. Hence,  the problem $(P)$ has at least two bounded weak solutions.
  \end{theorem} 
  The paper is organized as follows. In Section \ref{Section 2} we recall some preliminary results and show that the functional $I$ has a geometrical structure. In Section \ref{Section 3} we prove our main result, Theorem \ref{main result}.
  	
  \section*{Notation}	
  Through this paper, we use the following notations.
  	\begin{enumerate}
  		\item [1)] The Lebesgue norm $(\int_{\Omega} |u|^{p})^{\frac{1}{p}}$ in $L^{p}(\Omega)$ is denoted by $\|.\|_{p}$ for $p\in [1,+\infty[$. The norm in $L^{\infty}(\Omega)$ is denoted by $\|u\|_{L^{\infty}(\Omega)}:=ess \sup_{x\in\Omega}|u(x)|$. The H\"{o}lder conjugate of $p$ is denoted by $p^{\prime}$.
  		\item [2)]The spaces $W_{0}^{1,p}(\Omega)$ and $W^{-1,p'}(\Omega)$  are equipped with  Poincaré norm $\|u\|:=(\int_{\Omega} |\nabla u|^{p})^{\frac{1}{p}}$ and the dual norm $\|\cdot\|_{\ast}:=\|\cdot\|_{W^{-1,p'}(\Omega)}$ respectively.
  		\item [3)]We denote by $B(0,R)$ the ball of radius $R$ centered at $0$ in $W_{0}^{1,p}(\Omega)$ and  $\partial B(0,R) $ its boundary.
  		\item[4)] We denote by $C_{i}, c_{i}>0$ any positive constants that are not essential in the arguments and that may vary from one line to another. 
  	\end{enumerate}
  	\section{Preliminaries and geometry of the functional $I$}\label{Section 2}
  	In this section, we recall the standard definitions of Palais-Smale sequence at the level $\tilde{c}$ and Palais-Smale condition at the level $\tilde{c}$ for $I$, and we prove that the functional $I$ defined  in $(\ref{func})$ has a geometrical structure.
  	
  	 Let us define  the  level at $\tilde{c}$ as follows
  	$$\tilde{c}= \inf_{\gamma\in \Gamma} \max_{t\in [0,1]} I(\gamma(t)),$$ 
  	where 
  	$\Gamma =\{\gamma\in C([0,1], W_{0}^{1,p}(\Omega)): \gamma(0)=0,\; \gamma(1)=v_{0}\}$ is the set of continuous paths joining $0$ and $v_{0}$, where $v_{0} \in W_{0}^{1,p}(\Omega)$ is defined  in Proposition \ref{goem structural} below.
  	\begin{definition}\label{ Palais-Smale sequence}
  		Let $E$ be a Banach space with dual space $E^{\ast}$ and  $(u_{n})$ is a sequence in $E$. We say that $(u_{n})$ is a Palais-Smale sequence at the level $\tilde{c}$ for $I$ if 
  		$$ I(u_{n}) \to \tilde{c},\; \; \mbox{ and }\; \; \|I^{\prime}(u_{n})\|_{E^{*}}\to 0.$$
  	\end{definition}
  	\begin{definition}\label{ Palais-Smale condition}
  	We say that $I$ satisfies the Palais-Smale condition at the level $\tilde{c}$  if any  Palais-Smale sequence at the level $\tilde{c}$ for $I$ possesses a convergent subsequence.
  	\end{definition}
   In order to prove that $I$  has a geometrical structure, we need some properties of $g$, which we gather in  the following lemma without proof
   
   \begin{lemma}\label{properties of g}
   		~
   	\begin{enumerate}\label{limites}
   		\item $\frac{g(s)}{|s|^{p-2}s} \to c$ as $s\to 0$, where $c=0$  if \;$q>p-1$ and $c=1$ if $q=p-1$.\label{ voisinage de zer0}
   		\item $\frac{g(s)}{|s|^{q-1}s} \to (p-1)^{q-p+1}$  as $s\to 0$, for all $q>0$. \label{limit at zero}
   		\item  $\frac{g(s)}{s^{p-1}}\to +\infty$  \, and \, $\frac{G(s)}{s^{p}} \to +\infty$  as  $s\to +\infty$, for all  $q>0.$ \label{limit as s goes to infinity}
   	\end{enumerate}
   \end{lemma}
   
   \begin{lemma} \label{sublinear of g}
   	
   	~
   	\begin{enumerate}
   			
   		\item If $q\geq p-1$, then we have 
   	\begin{equation*}\label{sublinear1}
   	  	|g(s)|\leq c_{0}|s|^{r}+c_{1}|s|^{p-1},
   	\end{equation*}
for all   $s>-\frac{1}{\mu},$  and for all $ r\in (p-1,p).$
\vspace*{1.5mm}
   	\item If $0<q< p-1$, then we have
   	\begin{equation*}\label{sublinear2}
   	   		|g(s)|\leq c_{1}|s|^{r}+c_{2}|s|^{q},
   	\end{equation*}
   	for all  $s>-\frac{1}{\mu},$  and for all $ r\in (p-1,p).$
   	\end{enumerate}
   \end{lemma}
   \begin{proof}
   	By using  Lemma \ref{limites} (\ref{ voisinage de zer0}), there exists $\eta>0$ such that  for all $|s|<\eta$  we have
   	$$|g(s)|\leq c_{1}|s|^{p-1}.$$
   	Let $\delta\in (0,1)$. If $s\geq \eta$, then  we have
   	\begin{equation}\label{1.1 of Lemma 2.2}
   		g(s)\leq c_{2}(\eta,\mu, \delta)s^{p-1+\delta}.
   	\end{equation}
   	Moreover, simple calculation yield
   	$$g^{\prime}(s)=\frac{(p-1)^{q-p+1}}{\mu^{q-1}}(1+\mu s)^{p-2}|\ln(1+\mu s)|^{q-1}\left[ (p-1)\ln(1+\mu s)+q \right].$$ 
   	Now, if $-\frac{1}{\mu}<s\leq -\eta$, then we have
   	$|g(s)|\leq |g(T)|,$ where  $T=(e^{\frac{-q}{p-1}}-1)/\mu.$ Hence,
   	\begin{equation}\label{1.2 of Lemma 2.2}
   	|g(s)|\leq c_{3}(\eta,\mu, \delta)|s|^{p-1+\delta}.
   	\end{equation}
By combining (\ref{1.1 of Lemma 2.2}) and (\ref{1.2 of Lemma 2.2}),  (\ref{sublinear1}) holds. To prove the property (\ref{sublinear2}), we use   Lemma \ref{limites} (\ref{limit at zero}) and  the same previous argument.
   \end{proof}
  
   \begin{proposition} \label{goem structural} 
   	
   Assume that $(H)$ holds. If $\|c\|_{k}$ and $\|h\|_{k}$ are suitably small, then the functional $I$ has a geometrical structure, that is, $I$ satisfies the following properties
   	\begin{enumerate}
   		\item [i)] there exists $\rho>0$ such that for all $v$ in $\partial B(0,\rho)$,  $I(v)\geq \beta$, where $\beta>0$.  
   		\item [ii)] there exists  $v_{0}\in W_{0}^{1,p}(\Omega)$ such that   $\|v_{0}\|>\rho$ and  $I(v_{0})\leq 0$. 
   	\end{enumerate} 
   \end{proposition} 
   \begin{proof}
   	i)   To prove this lemma we distinguish two cases on $q$ . Firstly, if $0<q< p-1$, then by  using   Lemma \ref{sublinear of g} (\ref{sublinear2}) and H\"{o}lder's inequality,  we get 	
   	$$ \int_{\Omega}c(x)G(v)\leq c_{1}\|c\|_{k}\|v^{r+1}\|_{k^{\prime}}+ c_{2}\|c\|_{k}\|v^{q+1}\|_{k^{\prime}}.$$	
   	We choose $r> p-1$ with $r$ close to $p-1$  such that $(r+1)k^{\prime}< \frac{pN}{N-p}$, which exists due to the assumption $k>\frac{N}{p}$. Obviously, $(q+1)k^{\prime}< \frac{pN}{N-p}$. Thus,  by using  Sobolev's embedding we get	
   	
   	$$\int_{\Omega}c(x)G(v)\leq C_{1}\|c\|_{k}\|v\|^{r+1}+C_{2}\|c\|_{k}\|v\|^{q+1}.$$
   	Moreover, from the definition  of the function $f$ in (\ref{f}), we have 
   	\begin{equation}\label{growth of f}
   	|f(v)|\leq c(1+|v|^{p-1}), \mbox{ for some } c>0.
   	\end{equation} 
   	Using  Sobolev's embedding, we get
   	$$\int_{\Omega}h(x)F(v)\leq C_{3}\|h\|_{k}+C_{4}\|h\|_{k}\|v\|^{p}.$$	
   	By the definition   of $I$ in (\ref{func}), we deduce that
   	$$I(v)\geq \frac{1}{p} \|v\|^{p} -C_{1}\|c\|_{k}\|v\|^{r+1}-C_{2}\|c\|_{k}\|v\|^{q+1}-C_{3}\|h\|_{k}-C_{4}\|h\|_{k}\|v\|^{p}.$$
   	Now, let $v$ in $\partial B(0,\rho)$. Then, we have
   	$$I(v)\geq \frac{1}{p} \rho^{p}-\|c\|_{k}(C_{1}\rho^{r+1}+C_{2}\rho^{q+1})-\|h\|_{k}(C_{3}+C_{4}\rho^{p}).$$
   	We take $\rho$ sufficiently  large, and such that  $\|c\|_{k} \leq \rho^{-r-2+p}$ and  $\|h\|_{k}\leq \rho^{-1}$ \\(which are sufficiently small by hypothesis), then 
   	$$I(v)\geq \frac{1}{p}  \rho^{p} - C\rho^{p-1}\geq \rho^{p-1}\left(\frac{1}{p} \rho -  C\right)=\beta_{1}.$$
   	Secondly, that is $q\geq p-1$, we choose again $r$ as above such that $pk^{\prime}<(r+1)k^{\prime}< \frac{pN}{N-p}$. Then,  by using   Lemma \ref{sublinear of g} (\ref{sublinear1}) and Sobolev's embedding, we get
   	   	$$ \int_{\Omega}c(x)G(v)\leq c_{1}\|c\|_{k}\|v\|^{r+1}+ c_{2}\|c\|_{k}\|v\|^{p}.$$
   	   Now,  as the first case, we get 
   	   		$$I(v)\geq \frac{1}{p}  \rho^{p} - C^{\prime}\rho^{p-1}\geq \rho^{p-1}\left(\frac{1}{p} \rho -  C^{\prime}\right)=\beta_{2}.$$
   	   		Finally, we summarize the two cases and get
   	   			$$I(v)\geq \beta,\; \; \mbox{ where }\; \; \beta=\min(\beta_{1}, \beta_{2}).$$
   	ii) To prove the second property, we show that $I(tv)\to -\infty$ as $t\to +\infty$. For this, let $v \in C^{\infty}_{0}(\Omega) $ be a positive function such that $cv\gneqq0$. By the definition  of $I$ in (\ref{func}), we have  
   	\begin{align*}  	
   	I(tv)&=\frac{t^{p}}{p}\int_{\Omega}|\nabla v|^{p}-\int_{\Omega}c(x)G(tv)-\int_{\Omega}h(x)F(tv)\\
   	&=t^{p}\left(\frac{1}{p}\int_{\Omega}|\nabla v|^{p}-\int_{\Omega}c(x)\frac{G(tv)}{t^{p}v^{p}}v^{p}-\int_{\Omega}h(x)\frac{F(tv)}{t^{p}v^{p}}v^{p}\right).
   	\end{align*} 
   	From  inequality (\ref{growth of f}), we get 
   	$$\int_{\Omega}| h(x)\frac{F(tv)}{t^{p}v^{p}}v^{p}|\leq c \; \; \mbox{ as }\; \;  t\to +\infty.$$
   	Moerever, by   Lemma \ref{properties of g} (\ref{limit as s goes to infinity}), we get $$\int_{\Omega}c(x)\frac{G(tv)}{t^{p}v^{p}}v^{p}\to +\infty \; \; \mbox{ as }\; \;  t\to +\infty.$$
   	Thus, we deduce the desired result.
   \end{proof}
  Finally, we stress that since $I$ has a geometrical structure, then  the existence of a Palais-Smale sequence at the level $\tilde{c}$ for $I$  is ensured. This  can be observed directly from the proof given in (\cite{AR}), or alternatively using Ekeland's variational principle (\cite{Ekeland}).
  \section{Proof of Theorem \ref{main result}}\label{Section 3}
  Recall from introduction that the proof of our main result is divided into two steps as follows. In the first step, we show the existence of the first critical point for the $C^{1}$-functional $I$ by using the Mountain Pass Theorem due to Ambrosetti-Rabinowitz (\cite{AR}). Precisely,  we show that the functional  $I$ satisfies the  Palais-Smale  condition at the level $\tilde{c}$.  In the second step, we show the existence of the second critical point of $I$ on $B(0,\rho)$ (which is a local minimum) by using the lower semicontinuity argument. Moreover,   we are going to see that these  critical points are not the same.   Finally, we show that any solution of problem $(P)$ is bounded.

  \subsection{First critical point: Palais-Smale condition}\label{Palais-Smale condition}
  ~\newline

In this subsection, we prove that $I$ satisfies the  Palais-Smale  condition at the level $\tilde{c}$. Precisely, we show that any Palais-Smale sequence   at the level $\tilde{c}$ for $I$ is  bounded in $ W_{0}^{1,p}(\Omega)$, and then, it has  a strongly convergent subsequence.

 They key point  to prove the boundedness of the  Palais-Smale sequence at the level $\tilde{c}$ in  $W_{0}^{1,p}(\Omega)$, is to show that  $g$   verifies  the nonquadraticity condition at infinity $(NQ)$. Indeed, we have the following lemma 
\begin{lemma}\label{(NQ)}
	The function $g$ defined  in $(\ref{function g})$ verifies the nonquadraticity condition  at infinity ($NQ$); 
	$$(NQ)\ \;  \; \; \; H(s)=g(s)s-pG(s)\to +\infty, \;\mbox{ where }\; s\to +\infty. $$
\end{lemma}
\begin{proof}
	To prove ($NQ$), we show that $H$ is increasing and unbounded for $s$ sufficiently large. We recall that  $H(s)=g(s)s-pG(s)$. Then,  by  simple calculations, we get 
	$$H^{\prime}(s)=C\mu s(1+\mu s)^{p-2}(\ln(1+\mu s))^{q-1}[(1-p)\frac{\ln(1+\mu s)}{\mu s}+q],$$ 
	where $C=(p-1)^{q-p+1}/\mu^{q}.$
	Thus, $H$ is increasing for $s$ large enough. Moreover,  $H$ is unbounded. Indeed, by contradiction, if  $H$ is bounded, then  there exists a positive constant $M$ such that
	$$ \; \; H(s)\leq M,\; \; \mbox{ for s large enough}.$$
	In addition, from the definition of $H$ and using integration  by parts on $G$, we get 
	$$H(s)=-C\frac{1}{\mu}(\ln(1+\mu s))^{q}(1+\mu s)^{p-1}+qC\int_{0}^{s}(1+\mu t)^{p-1}(\ln(1+\mu t))^{q-1}dt.$$
	By choosing $\delta\in (p-1,p)$, we obtain 
	$$\frac{H(s)}{s^{\delta}}=-\frac{1}{\mu}\frac{(\ln(1+\mu s))^{q}(1+\mu s)^{p-1}}{s^{\delta}}+qC\frac{\int_{0}^{s}(1+\mu t)^{p-1}(\ln(1+\mu t))^{q-1}}{s^{\delta}}\leq \frac{M}{s^{\delta}}.$$
	When $ s\to +\infty$, we obtain  $\frac{H(s)}{s^{\delta}} \to +\infty$ and $ \frac{M}{s^{\delta}} \to 0$. Hence, we have a contradiction. As a conclusion, the function $g$ verifies $(NQ)$.
\end{proof}
\begin{lemma}\label{Cerami sequeq}
	 Let  $(u_{n})$ be a Palais-Smale sequence at the level $\tilde{c}$ for $I$ in $W_{0}^{1,p}(\Omega)$.  Then, $(u_{n})$  is bounded in $W_{0}^{1,p}(\Omega)$.	
\end{lemma}
\begin{proof}
	Let $(u_{n})$ be a  Palais-Smale sequence at the level $\tilde{c}$ for $I$ in $W_{0}^{1,p}(\Omega)$. We prove by contradiction that $(u_{n})$ is bounded in $W_{0}^{1,p}(\Omega)$. We assume that $(u_{n})$ is unbounded in $W_{0}^{1,p}(\Omega)$, that is, $\|u_{n}\|\to +\infty$. \\
	 For all integer $n\geq 0$,  we define
	$$I(z_{n}):=\max_{0\leq t\leq1 }I(t u_{n}),\; \; \mbox{ where } \; \; z_{n}=t_{n}u_{n} \; \mbox{ and } t_{n}\in [0,1].$$
	We are going to  prove that $I(z_{n})\to +\infty$ and also  $(I(z_{n}))$ is bounded, which is  the  desired contradiction.
	
	\vspace{1em}
	\noindent {\bfseries a) Showing that \boldmath$I(z_{n})\to +\infty$} : \hspace*{0.5em}
		We set $v_{n}:= \frac{u_{n}}{\|u_{n}\|}$,  then $(v_{n})$ is bounded in $W_{0}^{1,p}(\Omega)$. Hence, there exists  a subsequence denoted again   $(v_{n})$ such that $v_{n}$ converges weakly and strongly  to $ v$ in $W_{0}^{1,p}(\Omega)$ and in  $L^{s}(\Omega)$  for some  $1\leq s<p^{\ast}$ respectively. Moreover, $v_{n}$ also  converges to  $v$ almost everywhere in $\Omega$. Recall that $p^{\ast}:=\frac{Np}{N-p}$, is Sobolev conjugate.
		
\vspace*{1.5mm}
\textit{Now, we claim by contradiction that  $v\equiv 0$  a.e. in $\Omega$. }

	\vspace*{1.5mm}
		Since $(u_{n})$ is Palais-Smale type sequence, then  we have
		\begin{equation}
					I(u_{n}) \to \tilde{c}\; \; \mbox{ and }\; \; \; \; \|I^{\prime}(u_{n})\|_{*}\to 0.
		\end{equation}
Hence,
		\begin{equation}\label{equation1}
		\int_{\Omega}|\nabla u_{n}|^{p-2}\nabla u_{n}\nabla \varphi- \int_{\Omega}c(x)g(u_{n})\varphi- \int_{\Omega}h(x)f(u_{n})\varphi = \epsilon_{n},
		\end{equation} 
		for all  $\varphi \in W_{0}^{1,p}(\Omega)$ and  for some $\epsilon_{n}\to 0$ as $n\to +\infty$. We  divide both sides  of (\ref{equation1}) by $\|u_{n}\|^{p-1}$, to obtain  
		\begin{equation}\label{v=0}
		\int_{\Omega}c(x)\frac{g(u_{n})}{\|u_{n}\|^{p-1}}\varphi=\frac{\epsilon_{n}}{\|u_{n}\|^{p-1}}+\int_{\Omega}|\nabla v_{n}|^{p-2}\nabla v_{n}\nabla \varphi+ \int_{\Omega}h(x)\frac{f(u_{n})}{\|u_{n}\|^{p-1}}\varphi.
		\end{equation}
		On the one hand, since $v_{n}$ converges weakly to $v$ in $W_{0}^{1,p}(\Omega)$ and by  the inequality (\ref{growth of f}), then     for $n$ large enough the second  and the third terms of the right-hand side of (\ref{v=0}) are bounded.\newline
		On the other hand, if $v \not\equiv 0$ in $\Omega$, then $cv\not\equiv 0$ in  $\Omega$. Now, we choose $\varphi \in W_{0}^{1,p}(\Omega)$ such that $cv\varphi>0$ in $\Omega_{\varphi}$
		and $cv\varphi\equiv0$ in $\Omega\backslash \Omega_{\varphi}$, with $|\Omega_{\varphi}|>0$. Since $v_{n}\|u_{n}\| = u_{n} $ in $\Omega$, then by using Lemma \ref{properties of g}  (\ref{limit as s goes to infinity}), we obtain 
		$$\liminf c(x)\frac{g(u_{n})}{\|u_{n}\|^{p-1}}\varphi=\liminf c(x)(v_{n})^{p-1}\frac{g(v_{n}\|u_{n}\|)}{(v_{n}\|u_{n}\|)^{p-1}}\varphi=+\infty \; \; \mbox{in }\; \; \Omega_{\varphi} .$$
		Hence, by  using the Fatou's lemma in $(\ref{v=0})$ we obtain the unbounded term in the left-hand side of $(\ref{v=0})$. Hence, the claim ($i.e$ $v\equiv 0$  a.e. in $\Omega$.)

		Since $\|u_{n}\|\to +\infty$, then there exists $M>0$ such that $\|u_{n}\|>M$, for   $n$ large enough. Moreover, we have 
		$$I(z_{n})\geq I\left(M \frac{u_{n}}{\|u_{n}\|}\right)=I(M v_{n})=\frac{M^{p}}{p}- \int_{\Omega}c(x)G(Mv_{n})-\int_{\Omega}h(x)F( Mv_{n})   .$$
		In what follows, we treat only the case $0<q< p-1$. The other case follows with similar arguments.  From  Lemma \ref{sublinear of g} (\ref{sublinear2}),  we have  $|G(s)|\leq c_{1}|s|^{r+1}+c_{2}|s|^{q+1}$, where $p-1<r<p$. Since $c \in L^{k}(\Omega),$ for some $ k>\frac{N}{p}$ and $v_{n}$ converges strongly  to $v$ in $L^{s}(\Omega)$ with $1\leq s<p^{\ast}$, then, we obtain 
		$$\int_{\Omega}c(x)G(Mv_{n})\to 0\; \mbox{ as } n\to +\infty,$$
		 due to  $v\equiv0  \mbox{ a.e. in } \Omega.$
		By H\"{o}lder's inequality, we get
		$$\int_{\Omega}h(x)F( Mv_{n})\leq C \;\mbox{ as }\; n\to +\infty.$$ 
		Hence, by choosing $M>0$ large enough,  we deduce that  $I(z_{n})\to +\infty$, as $n\to +\infty.$

		\vspace{1em}
		\noindent {\bfseries b) Showing that \boldmath$I(z_{n})$ is bounded} : \hspace*{0.5em}
		To prove that $(I(z_{n}))$ is bounded, we distinguish two cases:  $t_{n}\leq \frac{2}{\|u_{n}\|}$ and $t_{n}> \frac{2}{\|u_{n}\|}$. 
		
		
		\vspace*{1em}
		\noindent\textit{The case $t_{n}\leq \frac{2}{\|u_{n}\|}$: }

		Here, we only handle the proof for $ q \in (0,p-1)$. The other case follows as in the proof of Proposition \ref{goem structural} i).
		By the definition of $(z_{n})$ and $I\in  C^{1}(W_{0}^{1,p}(\Omega),\mathbb{R})$, we have $\langle I^{\prime}(t_{n}u_{n}), t_{n}u_{n} \rangle=0$, which means that
		$$t^{p}_{n}\| u_{n}\|^{p}=\int_{\Omega}c(x)g(t_{n}u_{n})t_{n}u_{n}+\int_{\Omega}h(x)f(t_{n}u_{n})t_{n}u_{n}.$$
		By the definition  of $I$ in $(\ref{func})$, we have 
		\begin{equation}
		\begin{aligned}\label{CASE}
		pI(t_{n}u_{n})=& t^{p}_{n}\| u_{n}\|^{p} -p \int_{\Omega}c(x)G(t_{n}u_{n})-p\int_{\Omega}h(x)F(t_{n}u_{n})\\
		=& \int_{\Omega}c(x)H(t_{n}u_{n})+ \int_{\Omega}h(x)K(t_{n}u_{n}),
		\end{aligned}
		\end{equation}
		where the function  $H$ is defined  in $(NQ)$ and $K(s):=f(s)s-pF(s)$.
		Moreover,  from  Lemma \ref{sublinear of g} (\ref{sublinear2}), we have
		\begin{align*}
		\int_{\Omega}c(x)H(t_{n}u_{n})&\leq\int_{\Omega}\lvert c(x)\rvert\lvert g(t_{n}u_{n}) t_{n}u_{n}\rvert+p\int_{\Omega}\lvert c(x)\rvert\lvert G(t_{n}u_{n})\rvert\\
		&\leq c_1\int_{\Omega}\lvert c(x)\rvert\lvert t_{n}u_{n}\rvert^{r+1}+c_2\int_{\Omega}\lvert c(x)\rvert\lvert t_{n}u_{n}\rvert^{q+1}.
		\end{align*}
		By choosing $r$ and $q$ as in the proof of Proposition \ref{goem structural} i), we get
		\begin{equation}
		\begin{aligned}
		\int_{\Omega}c(x)H(t_{n}u_{n}) &\leq C_1\lVert c\rVert_{k}\lVert t_{n}u_{n}\rVert^{r+1}+C_2\lVert c\rVert_{k}\lVert t_{n}u_{n}\rVert^{q+1}.\\
		\end{aligned}\label{case1}
		\end{equation}
		By inequality (\ref{growth of f}) and  Sobolev's embedding, we get 
		\begin{equation}
		\begin{aligned}
		\int_{\Omega}h(x)K(t_{n}u_{n})&\leq\int_{\Omega}\lvert h(x)\rvert \lvert f(t_{n}u_{n})t_{n}u_{n}\rvert+p\int_{\Omega}\lvert h(x)\rvert \lvert(F(t_{n}u_{n})t_{n}u_{n}\rvert\\
		&\leq c_1\lVert h\rVert_{k}+c_2\lVert h\rVert_{k}\lVert t_{n}u_{n}\rVert+c_3\lVert h\rVert_{k}\lVert t_{n}u_{n}\rVert^{p}.
		\end{aligned}\label{case1.2}
		\end{equation}
		 Then, by $(\ref{CASE})$, $(\ref{case1})$, and $(\ref{case1.2})$, we obtain 
		$$I(t_{n}u_{n})\leq C,$$
		for all  $n\geq0,$ where $C$ is independent of $n$. Thus, $(I(z_{n}))$ is bounded, which  contradicts  the fact that $(I(z_{n}))$ is unbounded (see \textbf {a)}).
			
		\vspace*{1em}
		\noindent\textit{The case $t_{n}> \frac{2}{\|u_{n}\|}$: }

		  Here, we are proceeding the technique inspired by \cite{Marcelo}. To this end, we need the following technical lemma
		 \begin{lemma}\label{T. lemma}
		 	Let $\varPhi: \mathbb{R}\to \mathbb{R}$ the nonnegative function defined as
		 	$$ \varPhi(s)=
		 	\begin{cases}
		 	e^{-\epsilon/s^{2}},& \mbox{ if  } \; \; s\neq 0,\\
		 	0, & \mbox{ if } \; \;  s=0,
		 	\end{cases}$$
		 	with $\epsilon>0$. Then, we have
		 	\begin{enumerate}
		 		\item [i)] $\lim\limits_{s\to 0}\varPhi(s)=\lim\limits_{s\to 0}\varPhi^{\prime}(s)=0.$
		 		\item [ii)] 	for any positive function $z$ in $\Omega$ and $p>1$,\vspace*{-0.2em}
		 		$$ \lim\limits_{\epsilon \rightarrow 0}\int_{\Omega}\int_{s}^{t}\frac{z(x)}{\tau^{p+1}}\left(\dfrac{1-\Phi_{\epsilon}(|\tau u_{n}|)}{\|u_{n}\|^{p}}\right)d\tau dx=0,\;  \mbox{ uniformly in } n\in\mathbb{N}.$$
		 	 
		 	\end{enumerate}	
		 \end{lemma}
		 \begin{proof}
		 	Obviously we have i). To prove  ii), we  
		 	follow the same  approach given in \cite{Marcelo} for the case $p=2$ and $z(x)=1$, which can be immediately generalized for  any positive function $z$ and $p>1$.
		 \end{proof}	
		 	Now,  we resume the proof of Lemma \ref{Cerami sequeq}. From  Lemma \ref{(NQ)}, we have
			$H(s)\geq \sigma$, for $s$ large enough and some $\sigma>0$ (which will be chosen later). Moreover, if $0<q<p-1$, then from   Lemma \ref{limites} (\ref{limit at zero}), we have for $s$ sufficiently small, 
			$$H(s)\geq-C_{1}|s|^{q+1}.$$
			Then, by the continuity of $H$, we have  for all $s>-\frac{1}{\mu}$, 
			\begin{equation}\label{Brazil}
			H(s)\geq \sigma\Phi_{\epsilon}(s)-C_{2}|s|^{q+1}.
			\end{equation}
			Let  $0<s<t$, then we have
			\begin{align} \nonumber
			\dfrac{I(tu_{n})}{t^{p}\|u_n\|^{p}}-\dfrac{I(su_{n})}{s^{p}\|u_{n}\|^{p}}&=-\int_{\Omega}c(x)\left[\dfrac{G(tu_{n})}{t^{p}\|u_n\|^{p}}-\dfrac{G(su_{n})}{s^{p}\|u_n\|^{p}}\right]\\  \nonumber
			&-\int_{\Omega}h(x)\left[\dfrac{F(tu_{n})}{t^{p}\|u_n\|^{p}}-\dfrac{F(su_{n})}{s^{p}\|u_n\|^{p}}\right]\\ \label{3.8}
			&=\underbrace{-\int_{\Omega}c(x)\int_{s}^{t}\frac{d}{d\tau}\left(\dfrac{G(\tau u_{n})}{\tau^{p}\|u_n\|^{p}}\right)d\tau dx}_{A}\\ \nonumber &+\underbrace{\int_{\Omega}-h(x)\left[\dfrac{F(tu_{n})}{t^{p}\|u_n\|^{p}}-\dfrac{F(su_{n})}{s^{p}\|u_n\|^{p}}\right]}_{B}. 
			\end{align}
			Let us  handle the two  terms $A$ and $B$  respectively.
			\begin{align*}
			A&= -\int_{\Omega}\int_{s}^{t}c(x)\frac{\tau^{p}u_{n}g(\tau u_{n})-p\tau^{p-1}G(\tau u_{n})}{\tau^{2p}\|u_{n}\|^{p}}d\tau dx\\
			&=-\int_{\Omega}\int_{s}^{t}\frac{c(x)}{\|u_{n}\|^{p}}\frac{H(\tau u_{n})}{\tau^{p+1}}d\tau dx.
			\end{align*}
			By using $(\ref{Brazil})$, we get
			\begin{align}\label{A}
			A&\leq \int_{\Omega}\int_{s}^{t}\frac{c(x)}{\|u_{n}\|^{p}}\left(C_{2}\frac{|u_{n}|^{q+1}}{\tau^{p-q}}-\sigma\frac{\Phi_{\epsilon}(|\tau u_{n}|)}{\tau^{p+1}}\right)d\tau dx\\\nonumber 
			&\leq\int_{\Omega}\frac{c(x)}{\|u_{n}\|^{p}}\left(\frac{C_{2}}{p-q-1}\frac{|u_{n}|^{q+1}}{s^{p-q-1}}-\sigma\int_{s}^{t}\frac{\Phi_{\epsilon}(|\tau u_{n}|)}{\tau^{p+1}}d\tau\right)dx.
			\end{align}
				For the term $B$, we have
				\begin{align} \label{B}
				B&\leq C\left(  \int_{\Omega}|h(x)|\frac{(1+|tu_{n}|)^{p}}{t^{p}\|u_{n}\|^{p}}+\int_{\Omega}|h(x)|\frac{(1+|su_{n}|)^{p}}{s^{p}\|u_{n}\|^{p}}\right)\\ \nonumber
				&\leq C\left(\int_{\Omega}|h(x)|\left(\frac{1}{t_{n}\|u_{n}\|}+\frac{|u_{n}|}{\|u_{n}\|}\right)^{p}+\int_{\Omega}|h(x)|\left(\frac{1}{s_{n}\|u_{n}\|}+\frac{|u_{n}|}{\|u_{n}\|}\right)^{p}\right).
				\end{align}
				By setting  $s:=\frac{1}{\|u_{n}\|}$,  we obtain 
				\begin{align}  
				\dfrac{I(tu_{n})}{t^{p}\|u_n\|}&\leqslant \nonumber I(v_{n})+\int_{\Omega}c(x)\left(\frac{C_{2}}{p-q-1}|v_{n}|^{q+1}-\sigma\int_{s}^{t}\frac{\Phi_{\epsilon}(|\tau u_{n}|)}{\tau^{p+1}\|u_{n}\|^{p}}\right)d\tau dx\\ \nonumber
				&+C\left(\int_{\Omega}|h(x)|(\dfrac{1}{2}+|v_{n}|)^{p}+\int_{\Omega^{+}}|h(x)|(1+|v_{n}|)^{p}\right)\\ \nonumber
				&\leqslant I(v_{n})+C\left[\int_{\Omega}c(x)|v_{n}|^{q+1}+\int_{\Omega}|h(x)|+2\int_{\Omega}|h(x)||v_{n}|^{p}\right]\\ \nonumber
				&-\sigma\int_{\Omega}\frac{c(x)}{p}\left(1-\dfrac{1}{t^{p}_{n}\|u_{n}\|^{p}}\right)+\sigma\int_{\Omega}\frac{c(x)}{p}\left(1-\dfrac{1}{t^{p}_{n}\|u_{n}\|^{p}}\right) \\ \nonumber
				& \qquad -\sigma\int_{\Omega}\int_{s}^{t}c(x)\frac{\Phi_{\epsilon}(|\tau u_{n}|)}{\tau^{p+1}\|u_{n}\|^{p}}d\tau dx\\ \nonumber
				&\leqslant I(v_{n})+C\left[\int_{\Omega}c(x)|v_{n}|^{q+1}+\int_{\Omega}|h(x)|+2\int_{\Omega}|h(x)||v_{n}|^{p}\right]\\ \nonumber
				&-\sigma\int_{\Omega}\frac{c(x)}{p}\left(1-\dfrac{1}{t^{p}_{n}\|u_{n}\|^{p}}\right)
				-\sigma\int_{\Omega}\int_{s}^{t}\frac{c(x)}{\tau^{p+1}}\left(\dfrac{1-\Phi_{\epsilon}(|\tau u_{n}|)}{\|u_{n}\|^{p}}\right)d\tau dx.\\ \nonumber
				\end{align}
				By  the technical Lemma \ref{T. lemma}, we have  
				$$\lim\limits_{\epsilon \rightarrow 0}\int_{\Omega}\int_{s}^{t}\frac{c(x)}{\tau^{p+1}}\left(\dfrac{1-\Phi_{\epsilon}(|\tau u_{n}|)}{\|u_{n}\|^{p}}\right)d\tau dx=0,\; \; \mbox{ uniformly in } n\in\mathbb{N}.$$ 
				Then,
				\begin{align*}
				\dfrac{I(tu_{n})}{t^{p}\|u_n\|^{p}} &\leqslant \frac{1}{p}-\int_{\Omega}c(x)G(v_{n})-\int_{\Omega}h(x)F(v_{n})+C\left[\int_{\Omega}c(x)|v_{n}|^{q+1}\right.\\
				&\left.+\int_{\Omega}|h(x)| +2\int_{\Omega}|h(x)||v_{n}|^{p}\right]-\sigma\int_{\Omega}\frac{c(x)}{p}\left(1-\frac{1}{2^{p}}\right).
				\end{align*}
				We choose $\sigma$ such that
				$$\sigma>\dfrac{2^{p}(1+pC\|h\|_{k})}{(2^{p}-1)\int_{\Omega}c(x)}.$$
				which gives,
				$$\frac{1}{p} +C\|h\|_{k}-\sigma\int_{\Omega}\frac{c(x)}{p}\left(1-\frac{1}{2^{p}}\right)dx<0.$$   
				Since $v_{n}$ converges to $0$ almost everywhere in $\Omega$, weakly in $W^{1,p}_{0}(\Omega)$, and strongly in $L^{s}(\Omega)$ for some  $1\leq s<p^{\ast}$,
				then, we have 
				$$I(t_{n}u_{n})<0, \mbox{ in } \Omega \mbox{ for } n \mbox{ large enough} .$$
				Hence, $(I(z_n))$ is bounded. Therefore, this contradicts the fact that $(I(z_n))$ is unbounded (see \textbf{a)}).\\
				
				Now, If $q\geq p-1$, then from  Lemma \ref{limites} (\ref{ voisinage de zer0}) and the continuity of $H(s)$,  we have  for all $s>-\frac{1}{\mu}$,  
				\begin{equation}\label{Brazil2}
				H(s)\geq \sigma\Phi_{\epsilon}(s)-C_{1}|s|^{p-1}.
				\end{equation}
				Following the computations as in (\ref{3.8}), we find exactly the same terms $A$ and $B$. The term $B$ is handled as in (\ref{B}), whereas $A$ is handled as follows 
				\begin{align*}
				A&\leq \int_{\Omega}\int_{s}^{t}\frac{c(x)}{\|u_{n}\|^{p}}\left(C_{1}\frac{|u_{n}|^{p-1}}{\tau^{2}}-\sigma\frac{\Phi_{\epsilon}(|\tau u_{n}|)}{\tau^{p+1}}\right),\\
				&\leq \int_{\Omega}\int_{s}^{t}\frac{c(x)}{\|u_{n}\|^{p}}\left(C_{1}\frac{|u_{n}|^{p-1}}{s}-\sigma\frac{\Phi_{\epsilon}(|\tau u_{n}|)}{\tau^{p+1}}\right).
				\end{align*}
			Moreover, since $(p-1)k^{\prime}<pk^{\prime}<\frac{Np}{N-p}$, then by using Sobolev embedding, the rest of the proof is similar to the  case $q\in (0,p-1)$. Hence, we have also  the contradiction with the fact that $I$ is unbounded (see \textbf{a)}).
		\end{proof}
		To finish the proof of the Palais-Smale condition for $I$,   we only need to show the following lemma
		\begin{lemma}\label{Strongly converge}
			Any Palais-Smale sequence at the level $\tilde{c}$ of  $W_{0}^{1,p}(\Omega)$ has a strongly convergent subsequence.  	
		\end{lemma}	
		\begin{proof}
			Let $(u_{n})$ be a Palais-Smale sequence at the level $\tilde{c}$, then $I^{\prime}(u_{n})\to 0$ in $W^{-1,p^{\prime}}(\Omega)$, which means that $$-\Delta_{p}u_{n}- c(x)g(u_{n})- h(x)f(u_{n})\to 0 \; \mbox{ in } \; W^{-1,p^{\prime}}(\Omega).$$
			By  Lemma \ref{Cerami sequeq},  $(u_{n})$ is bounded in $W_{0}^{1,p}(\Omega)$. Hence, $u_{n}$ converges weakly to $u$   in $W_{0}^{1,p}(\Omega)$ and strongly  in $L^{s}(\Omega)$  for some  $1\leq s<p^{\ast}$. Therefore,
			\begin{equation}\label{converge}
			-\Delta_{p}u_{n}\to c(x)g(u)+ h(x)f(u)\; \mbox{ in } \; W^{-1,p^{\prime}}(\Omega).
			\end{equation}
			We know that the operator $-\Delta_{p}: W_{0}^{1,p}(\Omega) \mapsto  W^{-1,p^{\prime}}(\Omega)$
			is a homeomorphism
			( \cite{J. Mawhin}). Hence, from $(\ref{converge})$ we get $$u_{n}\to (-\Delta_{p})^{-1}( c(x)g(u)+h(x)f(u)) \;\mbox{ in }\; W_{0}^{1,p}(\Omega).$$
			Therefore, by the uniqueness of the limit we have 
			$$u_{n} \to u, \mbox{ in } W_{0}^{1,p}(\Omega).$$ 	
		\end{proof}
		\subsection{Second critical point }
		~\newline
		
		In this subsection,  we use the geometrical structure of $I$ (see Proposition \ref{goem structural}) and  the standard lower semicontinuity argument,   we show the existence of the second critical point.    We state the result as follows
		\begin{theorem}
						Assume that $\|c\|_{k}$ and $\|h\|_{k}$ are suitably small to ensure Proposition \ref{goem structural}. Then, the functional $I$ possesses a critical point $v\in B(0,\rho)$ with $I(v)\leq 0$.
		\end{theorem}
		\begin{proof}
			Since $I(0)=0$, then $\inf_{v\in B(0,\rho)}I(v)\leq 0$. Moreover, if $h \not\equiv 0$, then we obtain that  $\inf_{v\in B(0,\rho)}I(v)<0$. Indeed, we choose $v\in C^{\infty}_{0}(\Omega)$ a positive function that satisfies $cv>0$ and $hv>0$. From the definition of $I$ in (\ref{func}), we have for $t>0$  
			\begin{equation}\label{Final equaltion}
						I(tv)=t^{p}\left(\frac{1}{p}\int_{\Omega}|\nabla v|^{p}-\int_{\Omega}c(x)\frac{G(tv)}{t^{p}v^{p}}v^{p}-\int_{\Omega}h(x)\frac{F(tv)}{t^{p}v^{p}}v^{p}\right).
			\end{equation}
			If $q\geq p-1$, then from   Lemma \ref{properties of g} (\ref{limit at zero}), we have $G(s)/s^{p}\to c<+\infty$ as $s\to 0^{+}$.
			If $0<q< p-1$, obviously, we have $G(s)/s^{p}\to +\infty$ as $s\to 0^{+}$. In addition, in both cases, we have $\frac{F(s)}{s^{p}}\to +\infty$ as $s\to 0^{+}$. Hence, by using these limits, we get from (\ref{Final equaltion})  that $I(tv)<0$ for $t>0$ small enough. \\
			
			Now, we set $m:=\inf_{v\in B(0,\rho)}I(v)$. Then,  by Proposition \ref{goem structural} i), we have $I(v)\geq \beta>0$ for $\|v\|=\rho$. Moreover, there exists a sequence  $(v_{n})\subset B(0,\rho)$ such that $I(v_{n})$ converges to $m$. Since $(v_{n})$ is bounded in $W^{1,p}_{0}(\Omega)$, then there exists a subsequence denoted again $(v_{n})$ such that $v_{n}$ converges to $v$ weakly in $W^{1,p}_{0}(\Omega)$  and strongly in $L^{s}(\Omega)$ for some  $1\leq s<p^{\ast}$ respectively. Hence, we get
			$$\int_{\Omega}h(x)F(v_{n})\to \int_{\Omega}h(x)F(v)  \; \; \mbox{and }\; \; \int_{\Omega}c(x)G(v_{n})\to \int_{\Omega}c(x)G(v)    \; \; \mbox{ as }\; \;  n\to +\infty.$$
			In addition, since $\|v\|^{p}\leq\liminf_{n\to\infty}\|v_{n}\|^{p}$, then  $I(v)\leq m=\inf_{v\in B(0,\rho)}I(v)$. Hence, we conclude that $v$ is a local minimum of $I$ in $B(0,\rho)$.

		\end{proof}
		\begin{remark}
			By the  subsection \ref{Palais-Smale condition}, $I$ has a critical point at the level $\tilde{c}$, that is, there exists $w$ in $W_{0}^{1,p}(\Omega)$ such that 
			$I(w)=\tilde{c}$ and $I^{\prime}(w)=0$. Since $I(w)=\tilde{c}>0\geq I(v)$, where $v\in B(0,\rho)$ is the second critical point given in the previous theorem, then $w$ is different from $v$. Hence, we have two distinct solutions for the problem $(P)$.
		\end{remark}
		\subsection{ Boundedness of solutions}
		~\newline
		
		Now, to finish the proof of our main result, it remains to show the boundedness of the solutions. Therefore, we show  the following result
		\begin{proposition}
			Any solution $u$ of the problem $(P^{\prime})$ belongs to   $L^{\infty}(\Omega)$. 
		\end{proposition}
		\begin{proof}
			If $|u|\leq 1$, it is over. Otherwise,  we begin by writing the problem $(P^{\prime})$ as follows  
			$$-\Delta_{p}u=a(x)(1+|u|^{p-1}),$$ 
			where
			$$a(x)=\frac{ c(x)g(u) +h(x)f(u)}{1+|u|^{p-1}}.$$
			Then, by Theorem 2.4 in \cite{pucci}, we can deduce the boundedness of $u$ if we show  that $a$ belongs to $ L^{\frac{p}{N(1-\epsilon)}}(\Omega)$, for some $\epsilon \in ]0,1 [$. Indeed, from (\ref{growth of f}) and  Lemma \ref{sublinear of g}, we obtain
			\begin{equation}\label{pucci}
			|a(x)| \leq C\left[|c(x)|(|u|^{r-p+1}+1)+|h(x)|\right].
			\end{equation}
			Let $m>1$ and $m^{\prime}$ it's conjugate.  By using  H\"{o}lder's inequality in  $(\ref{pucci})$, we obtain 	
			$$ 
			\int_{\Omega}|a(x)|^{\frac{p}{N(1-\epsilon)}} \leq C\left[\|c(x)^{\frac{p}{N(1-\epsilon)}}\|_{m}\|u^{(r-p+1)\frac{p}{N(1-\epsilon)}}\|_{m'}+\|h^{\frac{p}{N(1-\epsilon)}}\|_{m}+1\right].$$
			By choosing $0<\epsilon< 1-\frac{(N-p)(r-p+1)}{N^{2}}-\frac{p}{kN}$, we have 
			$$\frac{p}{N(1-\epsilon)}m\leq k \; \;\mbox{ and }\; \;  (r-p+1)\frac{p}{N(1-\epsilon)}m'<\frac{Np}{N-p}.$$
			Hence, the terms $\|c(x)^{\frac{p}{N(1-\epsilon)}}\|_{m}$, $\|h(x)^{\frac{p}{N(1-\epsilon)}}\|_{m}$, and  $\|u^{(r-p+1)\frac{p}{N(1-\epsilon)}}\|_{m'}$ are finite (recall that $c, h\in L^{k}(\Omega)$ for some $k>\frac{N}{p}$).  
		\end{proof}	
	\section*{Acknowledgments}
The author would like to express their gratitude for the
assistance and the encouragement provided by their Ph.D.
supervisor Professor Abderrahmane ELHACHIMI. They are
also grateful for their colleague Dr. Aymane EL FARDI for
the revision and the valuable suggestions for improving the
presentation of the paper.


\begin{thebibliography}{99} \frenchspacing
			
			
			
			\bibitem{Abdellaoui}  B. Abdellaoui, A. Dall'Aglio, I. Peral, 
			\newblock{\it Some remarks on elliptic problems with critical growth in the gradient,}
			\newblock J. Differential Equations, 222 (2006), no. 1, 21-62.
			
			\bibitem{AR}   A. Ambrosetti, P. H. Rabinowitz, 
			\newblock{\it Dual Variational Methods In Critical Point Theory And Applications,}
			\newblock J.  Funct. Anal, 14 (1973), no. 4, 349-381.
			
			
			\bibitem{Coster2} D. Arcoya, C. De Coster, L. Jeanjean, K. Tanaka,
			\newblock { \it Remarks on the uniqueness for quasilinear elliptic equations with quadratic growth conditions,}
			\newblock{ J. Math. Anal. Appl,  420 (2014), no. 1, 772-780.}
			
			\bibitem{Coster} D. Arcoya, C. De Coster, L. Jeanjean, K. Tanaka, 
			\newblock {\it Continuum of solutions for an elliptic problem with critical growth in the gradient,} 
			\newblock J. Funct. Anal, 268 (2015), no. 8, 2298-2335.
			
			\bibitem{Barles-Blanc}	G. Barles, A.P. Blanc, C. Georgelin, M. Kobylanski,
			\newblock{\it Remarks on the maximum principle for nonlinear elliptic PDE with quadratic growth conditions,}  
			\newblock Ann. Sc. Norm. Super. Pisa Cl. Sci, Série 4,  Tome 28 (1999), no. 3,  381-404.
			
			\bibitem{Murat2} G. Barles, F. Murat, 
			\newblock {\it Uniqueness and the maximum principle for quasilinear elliptic equations with quadratic growth conditions,} 
			\newblock  Arch. Ration. Mech. Anal, 133 (1995), no. 1, 77-101.
			
			
			\bibitem{Boccardo} L. Boccardo, F. Murat, J. P. Puel,
			\newblock{\it Existence de solutions faibles pour des équations elliptiques quasi-linaires à croissance quadratique,}  
			\newblock Nonlinear partial differential equations and their applications. Collège de France Seminar, Vol. IV (Paris, 1981/1982), Res. Notes in Math., vol. 84, Pitman, Boston, Mass.-London, 1983, pp. 19-73. MR716511 (84k:35064)
			
			\bibitem{Murat} L. Boccardo, F. Murat,  J.-P. Puel, 
			\newblock{\it $L^{\infty}$ estimate for some nonlinear elliptic partial differential equations and application to an existence result, } 
			\newblock SIAM J. Math. Anal, 23 (1992), no. 2, 326-333. 
			
			\bibitem{Puel}  L. Boccardo, F. Murat,  J. P. Puel,  
			\newblock {\it Résultats d'existence pour certains problèmes elliptiques quasilinéaires, }
			\newblock Ann. Sc. Norm. Super. Pisa Cl. Sci,  Série 4,  Tome 11 (1984), no. 2,  213-235.
			
			\bibitem{Costa} D.G. Costa, C.A. Magalh\~aes,
			\newblock{\it Variational elliptic problems which are nonquadratic at infinity,} 
			\newblock  Nonlinear Anal, 23 (1994), no. 11, 1401-1412.
			
			\bibitem{Coster3} C. De Coster, L. Jeanjean, 
			\newblock{\it Multiplicity results in the non-coercive case for an elliptic problem with critical growth in the gradient, } 
			\newblock  J. Differential Equations, 262 (2017), no. 10,  5231-5270.
			
			\bibitem{J. Mawhin} G. Dinca, P. Jebelean, J. Mawhin,  
			\newblock{\it Variational and topological methods for Dirichlet problems with $p$-Laplacian,}
			\newblock Port. Math, 58 (2001), no. 3, 339-378.
			
			\bibitem{Ekeland} I. Ekeland, 
			\newblock{\it On the variational principle,}
			\newblock  J. Math. Anal. Appl, 47 (1974), no. 2, 324-353.
			
			\bibitem{Ferone-Murat1}	V. Ferone, F. Murat, 
			\newblock{\it Quasilinear problems having quadratic growth in the gradient: an existence
				result when the source term is small,}
			\newblock  Equations aux dérivées partielles et applications, Gauthier-
			Villars,  Ed. Sci. Med. Elsevier, Paris, (1998), 497-515.
			
			\bibitem{Ferone-Murat2} V. Ferone, F. Murat, 
			\newblock{\it	Nonlinear problems having quadratic growth in the gradient: an existence
				result when the source term is small,}
			\newblock  Nonlinear Analysis: Theory, Methods and Applications,
			vol. 42, no. 7, pp. 1309-1326, 2000.
			
			\bibitem{Marcelo}  M. Furtado, E. D. Silva, 
			\newblock{\it	Superlinear elliptic problems under the nonquadraticity condition at infinity,}
			\newblock J. Proc. Roy. Soc. Edinburgh Sect. A,  145 (2015), no. 4, 779-790. 
			
			\bibitem{Chile} J. J García-Melián, L. Iturriaga, H. Quoirin,  
			\newblock{\it	A priori bounds and existence of solutions for slightly superlinear elliptic problems,} 
			\newblock	Adv. Nonlinear Stud, 15 (2015), no. 4, 923-938.
			
			\bibitem{Iturriaga} L. Iturriaga, S. Lorca, J. S\'{a}nchez, 
			\newblock{\it Existence and Multiplicity Results for the $p$-Laplacian with a $p$-Gradient Term,} 
			\newblock NoDEA Nonlinear Differential Equations Appl. 15 (2008), no. 6, 729-743.
			
			\bibitem{Ubilla} L. Iturriaga, S. Lorca, P. Ubilla,
			\newblock{\it A quasilinear problem without the Ambrosetti-Rabinowitz-type condition,}
			\newblock Proc. Roy. Soc. Edinburgh Sect. A, 140 (2010), no. 2, 391-398.
			
			
			
			\bibitem{Quoirin} L. Jeanjean, H. Ramos Quoirin,
			\newblock{\it Multiple solutions for an indefinite elliptic problem with critical growth in the gradient, }
			\newblock Proc. Amer. Math. Soc, 144 (2016), no. 2,  575-586.
			
			
			\bibitem{Sirakov2} L. Jeanjean, B. Sirakov, 
			\newblock{\it Existence and multiplicity for elliptic problems with quadratic growth in the gradient,}
			\newblock Comm. Partial Differential Equations, 38 (2013), no. 2, 244-264.
			
			\bibitem{Maderna} C. Maderna, C. Pagani, S. Salsa,
			\newblock{\it  Quasilinear elliptic equations with quadratic growth in the gradient,}
			\newblock J. Differential Equations, 97 (1992), no. 1,  54-70.
			
			\bibitem{Miyagaki}	O.H. Miyagaki and M.A.S. Souto, 
			\newblock{\it Supelinear problems without Ambrosetti-Rabinowitz growth condition,}
			\newblock J. Differential Equations,  245 (2008), no. 12, 3628-3638.
			
			\bibitem{Poretta} A. Porretta, 
			\newblock{\it The ergodic limit for a viscous Hamilton Jacobi equation with Dirichlet conditions,}
			\newblock Atti Accad. Naz. Lincei Cl. Sci. Fis. Mat. Natur. Rend. Lincei (9) Mat. Appl, 21 (2010), no. 1,  59-78.
			
			\bibitem{pucci} P. Pucci, R. Servadei, 
			\newblock{\it Regularity of weak solutions of homogeneous or inhomogeneous quasilinear elliptic equations,}
			\newblock Indiana Univ. Math. J, 57 (2008), no. 7, 3329-3363.
			
			
			\bibitem{Sirakov} B. Sirakov, 
			\newblock{\it	Solvability of uniformly elliptic fully nonlinear PDE, }
			\newblock Arch. Ration. Mech. Anal, 195  (2010), no. 2, 579-607.
			
			
			\bibitem{Zou} M. Willem, W. Zou,
			\newblock{\it On a Schrdinger equation with periodic potential and spectrum point zero,}
			\newblock Indiana Univ. Math. J, 52 (2003), no. 1, 109-132.
			
		\end{thebibliography}
	\end{document}